\documentclass[11pt, twoside, a4paper, reqno]{amsart}

\usepackage{amsfonts, amssymb, amscd, amsmath, amsthm}
\usepackage[utf8]{inputenc}
\usepackage{graphicx}
\usepackage{float}
\usepackage{multicol}
\allowdisplaybreaks 
\usepackage{mathtools}
\usepackage{amsmath, amssymb}
\usepackage{amscd}
\usepackage{xcolor}
\usepackage{tikz-cd} 
\usepackage{graphicx} 
\usepackage[charter]{mathdesign}
\usepackage{hyperref}
\usepackage{enumitem}
\allowdisplaybreaks 
\usepackage{mathtools}
\newtheorem{definition}{Definition}[section]
\newtheorem{theorem}[definition]{Theorem}
\newtheorem{lemma}[definition]{Lemma}

\newtheorem{proposition}[definition]{Proposition}
\theoremstyle{definition}
\newtheorem{remark}[definition]{Remark}

\newcommand{\la}{\left\langle}
\newcommand{\ra}{\right\rangle}

\newcommand{\ucpa}{\text{UCP} \big(\mathcal{A}, \mathcal{B}  (\mathcal{H}) \big)}

\newcommand{\ucpapg}{\text{UCP}^{G_\tau} \big(\mathcal{A}, \mathcal{B}  (\mathcal{H} ) \big)}

\newcommand{\cpa}{\text{CP} \big(\mathcal{A}, \mathcal{B}  (\mathcal{H} ) \big)}
\newcommand{\cba}{\text{CB} \big(\mathcal{A}, \mathcal{B} (\mathcal{H} ) \big)}

\newcommand{\cpapg}{\text{CP}^{G_\tau} \big(\mathcal{A}, \mathcal{B} (\mathcal{H}) \big)}

\newcommand{\pacom}{\pi \big ( \mathcal{A} \big)^\prime}

\title[Extreme points of Unital Completely Positive maps invariant under partial action]{Extreme points of Unital Completely Positive maps invariant under partial action}

\author[Kulkarni]{Chaitanya J. Kulkarni}
\address{Chaitanya J. Kulkarni, Indian Institute of Science  Education and Research (IISER) Mohali, Knowledge City, S.A.S Nagar, Punjab 140306, India.}
\email{chaitanyakulkarni58@gmail.com}

\author[Hossain]{Md Amir Hossain}
\address{Md Amir Hossain, The Institute of Mathematical Sciences, A CI of Homi Bhabha National Institute, 4th cross street, CIT Campus, Taramani, Chennai, 600113, India.}
\email{mdamirhossain18@gmail.com}

\subjclass[]{46A55, 46B22, 46L55, 46L08, 47L07}
\keywords{extreme points; barycentric decomposition; completely positive maps; partial actions}
\thanks{\emph{Corresponding author's email:}  mdamirhossain18@gmail.com}
\begin{document}

\maketitle

%%%%%%%%%%%%%%%%%%%%%%%%%%%%%%%%%%%%%%%%%%%%%%%%%%%%%%%%%%%%%%%%%%%%%%%%%%%%%%%%%%%%%%%
\begin{abstract}
The classical Choquet theorem establishes a barycentric decomposition for elements in a compact convex subset of a locally convex topological vector space. This decomposition is achieved through a probability measure that is supported on the set of extreme points of the subset. In this work, we consider a partial action $\tau$ of a group $G$ on a $C^\ast$-algebra $\mathcal{A}$. For a fixed Hilbert space $\mathcal{H}$, we consider the set of all unital completely positive maps from $\mathcal{A}$ to $\mathcal{B}(\mathcal{H})$ that are invariant under the partial action $\tau$. This set forms a compact convex subset of a locally convex topological vector space. To complete the picture of the barycentric decomposition provided by the classical Choquet theorem, we characterize the set of extreme points of this set.
\end{abstract}

%%%%%%%%%%%%%%%%%%%%%%%%%%%%%%%%%%%%%%%%%%%%%%%%%%%%%%%%%%%%%%%%%%%%%%%%%%%%%%%%%%%%%%%

\section{Introduction} \label{sec;Introduction}

Decomposing a given element into other well-understood structures is often the broad goal in the decomposition theory. One of the methods in this direction is to decompose elements within a compact and convex subset of a locally convex topological vector space using the set of extreme points of such subsets. This leads to the study of the set of extreme points of a compact and convex subset of a locally convex topological vector space. 
In \cite{Arveson1}, W. Arveson characterized the extreme points of various subsets of completely positive maps within the classical (linear) convexity framework. Following this, several authors have introduced various non-commutative analogues of convexity. For instance, the concept of $C^*$-convexity was explored in \cite{LP, FM}, and matrix convexity was examined in \cite{EW}.  In \cite{DK}, the authors introduced the concept of nc-convexity. Additionally, CP-convexity was introduced in 1993 in \cite{F}. More recently, in \cite{AS}, the authors introduced $P$-$C^\ast$-convexity, defined relative to a positive operator $P$ on a Hilbert space. 

However, throughout this article, by a convex structure we always mean convexity in the classical (linear) sense. In this framework, the classical Choquet theorem provides a decomposition in a compact convex subset of a locally convex space. 
\begin{theorem} [{\cite[Page 14]{Phlp}}] 
Suppose that $X$ is a metrizable compact convex subset of a locally convex space $E$ and that $x_0$ is an element of $X$. Then there is a probability measure $\mu$ on $X$ whose \textit{barycenter} is $x_0$, that is,
\begin{equation} \label{eqn; barycentric decomposition 1}
f(x_0) = \int_{X} f \, \mathrm{d} \mu   
\end{equation}
for every continuous linear functional $f$ on $E$ and $\mu$ is supported by the extreme points of $X$.
\end{theorem}
\noindent
The decomposition of $x_0$ obtained in Equation \eqref{eqn; barycentric decomposition 1} is a \textit{barycentric decomposition} of $x_0$. Since the set of continuous linear functionals on $E$ separates the points of $E$, it suffices to consider only such functionals in the statement of the Choquet theorem. Moreover, the important fact is that the probability measure $\mu$ obtained from the Choquet theorem is supported by the set of extreme points of $X$. Therefore, understanding the structure of the extreme points of \( X \) is essential for completing the decomposition outlined by the Choquet theorem.

In this article, we consider a group $G$, a Hilbert space $\mathcal{H}$ and a unital $C^\ast$-algebra $\mathcal{A}$. Let $\tau$ be a \textit{partial action} of $G$ on $\mathcal{A}$ (see Section~\ref{sec; RND} for the definition of partial action). We then consider the set of all unital completely positive maps from $\mathcal{A}$ to $\mathcal{B}(\mathcal{H})$, the algebra of all bounded linear operators on $\mathcal{H}$, that are invariant under the partial action $\tau$. This set is denoted by $\ucpapg$. The set $\ucpapg$ forms a compact (with respect to the BW-topology) and convex subset of the locally convex space $\cba$, consisting of all completely bounded maps from $\mathcal{A}$ to $\mathcal{B}(\mathcal{H})$. Now, by applying the Choquet theorem in this setting, for $X= \ucpapg$ and $E= \cba$, for a given $\phi_0 \in \ucpapg$, we get a probability measure $\mu$ on $\ucpapg$ whose barycenter is $\phi_0$, that is,  
\begin{equation*}
f(\phi_0) = \int_{X} f \, \mathrm{d} \mu   
\end{equation*}
for every continuous linear functional $f$ on $\cba$ and $\mu$ is supported by the extreme points of $\ucpapg$. Thus, to complete the picture of the barycentric decomposition obtained from the Choquet theorem, it remains to provide the structure of extreme points of the set $\ucpapg$. The main result (Theorem~\ref{thm;egiucp}) of this article characterizes the set of extreme points of the compact and convex set $\ucpapg$. In particular, the article \cite{BK3} characterizes the set of extreme points of $\ucpapg$ when $\tau$ is a (global) group action. 

\medskip

The article is organized as follows. It is divided into four sections. In Section~\ref{sec;Preliminaries}, we recall some important results from the theory of Stinespring’s and Paschke’s dilations of completely positive maps, as well as Radon--Nikodym derivative type results in this context. These results are used extensively in the next two sections. In Section~\ref{sec; RND}, we establish Radon--Nikodym type results in the setting of partial group actions. These results serve as the main technical tools for proving the main theorem of the article. Finally, Section~\ref{sec; main result} is aimed to characterize the extreme points of the compact and convex set $\ucpapg$. This characterization is presented within the framework of Stinespring’s and Paschke’s dilations of completely positive maps (see Theorem~\ref{thm;egiucp}). The proof of the main theorem in this article uses a technique similar to the one given in~\cite{BK1}.

%%%%%%%%%%%%%%%%%%%%%%%%%%%%%%%%%%%%%%%%%%%%%%%%%%%%%%%%%%%%%%%%%%%%%%%%%%%%%%%%%%%%%%%

\section{Preliminaries} \label{sec;Preliminaries}
In this section, we briefly review some important results related to completely positive maps. We recall Stinespring’s and Paschke’s dilations of completely positive maps, and then discuss the Radon--Nikodym derivatives of completely positive maps with respect to both dilations. The results recalled in this section are mainly borrowed from~\cite{Arveson1, Paschke, Paulsen}. 

%%%%%%%%%%%%%%%%%%%%%%%%%%%%%%%%%%%%%%%%%%%%%%%%%%%%%%%%%%%%%%%%%%%%%%%%%%%%%%%%%%%%%%%

\subsection{Stinespring's dilation of a completely positive map}
We start by recalling Stinespring's dilation theorem.

\begin{theorem}[{\cite[Theorem 4.1]{Paulsen}}]\label{thm;sd}
Let $\mathcal{A}$ be a unital $C^\ast$-algebra and $\mathcal{H}$ be a Hilbert space. Let $\phi : \mathcal{A} \rightarrow \mathcal{B}(\mathcal{H})$ be a completely positive map. Then there exists a Hilbert space $\mathcal{K}$, a bounded linear map $V : \mathcal{H} \rightarrow \mathcal{K}$, and a unital \(^*\)-homomorphism $\pi : \mathcal{A} \rightarrow \mathcal{B}(\mathcal{K})$ such that
\begin{equation*}
\phi(a) = V^*\pi(a)V \; \;  \text{for all $a \in \mathcal{A}$}. 
\end{equation*} 
Moreover, the set $\big \{ \pi(a)Vh : a \in \mathcal{A}, h \in \mathcal{H} \big \}$ spans a dense subspace of $\mathcal{K}$, and if $\phi$ is unital, then $V : \mathcal{H} \rightarrow \mathcal{K}$ is an isometry.
\end{theorem}

Corresponding to a completely positive map $\phi : \mathcal{A} \rightarrow \mathcal{B}(\mathcal{H})$, a triple $(\pi, V, \mathcal{K} )$ obtained as in Theorem~\ref{thm;sd} is called a minimal Stinespring representation for $\phi$. The following proposition states that for a given completely positive map $\phi$, a minimal Stinespring representation is unique up to a unitary equivalence.  

\begin{proposition}[{\cite[Proposition 4.2]{Paulsen}}] \label{prop;umsd}
Let $\mathcal{A}$ be a unital $C^\ast$-algebra, $\mathcal{H}$ be a Hilbert space and let $\phi : \mathcal{A} \rightarrow \mathcal{B}(\mathcal{H})$ be a completely positive map. Suppose $\big (\pi_i, V_i, \mathcal{K}_i \big)_{i = 1, 2}$ be two minimal Stinespring representations for $\phi$. Then there exits a unitary $U : \mathcal{K}_1 \rightarrow \mathcal{K}_2$ such that $UV_1 = V_2$ and $U \pi_1 U^* = \pi_2$. 
\end{proposition}

\noindent 
For a given completely positive map $\phi : \mathcal{A} \rightarrow \mathcal{B}(\mathcal{H})$ and  a minimal Stinespring representation  $\big (\pi, V, \mathcal{K} \big )$, in view of Proposition \ref{prop;umsd}, we call $\big (\pi, V, \mathcal{K} \big )$ as the minimal Stinespring representation for $\phi$. For a unital $C^\ast$-algebra $\mathcal{A}$ and a Hilbert space $\mathcal{H}$, we consider the set
\begin{equation*}
\cpa := \big \{ \phi : \mathcal{A} \rightarrow \mathcal{B}(\mathcal{H}) \; \; : \; \;  \phi \; \; \text{is a completely positive map} \big \}.
\end{equation*} 
Now define a partial order "$\leq$" on $\cpa$ by $\phi_1 \leq \phi_2$ \big (where $\phi_1, \phi_2 \in \cpa$ \big ), if $\phi_2 - \phi_1 \in \cpa$ . Then for each fixed $\phi \in \cpa$, consider the set
\begin{equation*}
[0, \phi] = \big \{ \theta \in \cpa \; \; : \; \; \theta \leq \phi \big \}.
\end{equation*} 
Let $\phi \in \cpa$ and $\big (\pi, V, \mathcal{K} \big )$ be the minimal Stinespring representation for $\phi$. Then for each fixed $T \in \pacom$ define a map 
$\phi_T : \mathcal{A} \rightarrow \mathcal{B}(\mathcal{H})$ by $\phi_T(\cdot) := V^* \pi (\cdot) T V$. Next, we recall the Radon--Nikodym type result for completely positive maps from~\cite{Arveson1}.

\begin{theorem} [{\cite[Theorem 1.4.2]{Arveson1}}] \label{thm;arnd}
Let $\phi \in \cpa$ and $\big (\pi, V, \mathcal{K} \big )$ be the minimal Stinespring representation for $\phi$. Then there exists an affine order isomorphism of the partially ordered convex set of operators $ \big \{T \in \pacom \; \; : \; \;  0 \leq T \leq  1_\mathcal{K} \}$ onto $[0, \phi]$, which is given by the map 
\begin{equation*}
T \mapsto \phi_T(\cdot) := V^* \pi (\cdot) T V.
\end{equation*}
\end{theorem}

%%%%%%%%%%%%%%%%%%%%%%%%%%%%%%%%%%%%%%%%%%%%%%%%%%%%%%%%%%%%%%%%%%%%%%%%%%%%%%%%%%%%%%%%%%%%

\subsection{Paschke's dilations of a completely positive map}
In this subsection, we recall some results related to the theory of Paschke's dilation of a completely positive map.
 
\begin{definition} [{\cite[Definition 2.1]{Paschke}}] \label{def;phm}
Let $\mathcal{B}$ be a unital $C^\ast$-algebra. Then a pre-Hilbert $\mathcal{B}$-module $\mathcal{X}$ is a right $\mathcal{B}$-module equipped with a conjugate-bilinear map $\la \cdot, \cdot \ra : \mathcal{X} \times \mathcal{X} \rightarrow \mathcal{B}$ satisfying:
\begin{enumerate}
\item $\la x, x \ra \geq 0$ for all $x \in \mathcal{X}$;
\item $\la x, x \ra = 0$ if and only if $x = 0$;
\item $\la x, y \ra = \la y, x \ra^*$ for all $x, y \in \mathcal{X}$;
\item $\la x\cdot b, y \ra = \la x, y \ra b$ for all $x, y \in \mathcal{X}$ and $b \in \mathcal{B}$.
\end{enumerate}	
\end{definition}

\noindent
For a pre-Hilbert $\mathcal{B}$-module $\mathcal{X}$, we define a norm on $\mathcal{X}$ by $ \big \| x \big \|_\mathcal{X} = \big \|  \la x, x \ra \big \|^{\frac{1}{2}}.$ 

\begin{definition} [{\cite[Definition 2.4]{Paschke}}] \label{def;hm}
Let $\mathcal{B}$ be a unital $C^\ast$-algebra. Then a pre-Hilbert $\mathcal{B}$-module $\mathcal{X}$ which is complete with respect to the norm $\| \cdot \|_\mathcal{X}$ is called a Hilbert $\mathcal{B}$-module. 
\end{definition}

%Every bounded linear operator on a Hilbert space has an adjoint. However, for a  Hilbert $\mathcal{B}$-module (where $\mathcal{B}$ is a unital $C^\ast$-algebra) $\mathcal{X}$, it is not the case, and hence we consider the following set of operators on the Hilbert $\mathcal{B}$-module $\mathcal{X}$.

\begin{definition} [{\cite{Paschke}}] \label{def;ao}
Let $\mathcal{B}$ be a unital $C^\ast$-algebra and let $\mathcal{X}$ be a $\mathcal{B}$-module. Then a bounded operator $T : \mathcal{X} \rightarrow \mathcal{X}$ is called an \textit{adjointable operator}, if there exists a bounded operator $T^* : \mathcal{X} \rightarrow \mathcal{X}$ such that $\la Tx, y \ra = \la x, T^*y \ra$ for all $x, y \in \mathcal{X}$.
\end{definition}

\noindent
For a Hilbert $\mathcal{B}$-module $\mathcal{X}$, we denote the collection of all \textit{adjointable operators} on $\mathcal{X}$ by $\mathcal{P} ( \mathcal{X} )$. The set $\mathcal{P} ( \mathcal{X} )$ forms a $C^\ast$-algebra (see \cite{Paschke}). Now consider a unital $C^\ast$-algebra $\mathcal{A}$ and a $\ast$-representation $\sigma: \mathcal{A} \rightarrow \mathcal{P}( \mathcal{X} )$. For $e \in \mathcal{X}$, define a map $\phi : \mathcal{A} \rightarrow \mathcal{B}$ by
\begin{equation*}
\phi(a) := \big \langle \sigma(a)e, e \big \rangle  \; \; \text{for each} \; \; a \in \mathcal{A}.
\end{equation*} 
Then $\phi$ is a completely positive map. The following theorem, proved in~\cite{Paschke}, states that every completely positive map from $\mathcal{A}$ into $\mathcal{B}$ arises in the same way as above.

\begin{theorem} [{\cite[Theorem 5.2]{Paschke}}] \label{thm;pd}
Let $\mathcal{A}$ and $\mathcal{B}$ be unital \(C^*\)-algebras, and let $\phi : \mathcal{A} \rightarrow \mathcal{B}$ be a completely positive map. Then there exists a Hilbert $\mathcal{B}$-module $\mathcal{X}$, a \(^*\)-representation $\sigma : \mathcal{A} \rightarrow \mathcal{P}\big ( \mathcal{X} \big )$, and $e \in \mathcal{X}$ such that $\phi(a) = \big \langle \sigma(a)e, e \big \rangle$ for all $a \in \mathcal{A}$. Moreover, the set $\big \{ \sigma(a)(eb) \; \; : \; \;  a \in \mathcal{A}, b \in \mathcal{B} \big \}$ spans a dense subspace of $\mathcal{X}$.
\end{theorem}
\noindent
The dilation obtained in Theorem \ref{thm;pd} is called Paschke's dilation for a completely positive map $\phi$, and we denote this dilation by the triple $ \big (\sigma, e, \mathcal{X} \big )$. The property that the set $\big \{ \sigma(a)(eb) \; \; : \; \;  a \in \mathcal{A}, b \in \mathcal{B} \big \}$ spans a dense subspace of $\mathcal{X}$ makes this dilation unique in the following sense. 

\begin{theorem} [{\cite[Theorem 2.8]{BK3}}] \label{thm;umpd}
Let $\mathcal{A}$ and $\mathcal{B}$ be unital \(C^*\)-algebras, and let $\phi : \mathcal{A} \rightarrow \mathcal{B}$ be a completely positive map. Suppose $ \big (\sigma_i, e_i, \mathcal{X}_i \big )_{i =1, 2}$ be two Paschke's dilation for $\phi$, as given in Theorem \ref{thm;pd}. Then there exists a Hilbert module isomorphism $W : \mathcal{X}_1 \rightarrow \mathcal{X}_2$ such that $We_1 = e_2$, and $W \sigma_1(a) W^\ast = \sigma_2(a)$ for all $a \in \mathcal{A}$.
\end{theorem}

\noindent
For a given completely positive map $\phi : \mathcal{A} \rightarrow \mathcal{B}$ and a Paschke's dilation  $ \big (\sigma, e, \mathcal{X} \big )$, in view of Theorem \ref{thm;umpd}, we call $ \big (\sigma, e, \mathcal{X} \big )$ the Paschke's dilation for $\phi$. 

Let $\mathcal{B}$ be a unital $C^\ast$-algebra, and let $\mathcal{X}$ be a pre-Hilbert $\mathcal{B}$-module. Define $\mathcal{X}^\prime$ as 
\begin{equation} \label{eq; X prime}
\mathcal{X}^\prime := \big \{ f: \mathcal{X} \rightarrow \mathcal{B} \; \; : \; \;  f \; \;  \text{is a bounded $\mathcal{B}$-module map} \big \}.
\end{equation} 
For each fixed $x \in \mathcal{X}$, we define a map $\hat{x} : \mathcal{X} \rightarrow \mathcal{B}$ as $\hat{x}(y) := \langle y, x \rangle$. If $\mathcal{X}^\prime = \hat{\mathcal{X}} := \{\hat{x} \; \; : \; \; x \in \mathcal{X} \}$, then $\mathcal{X}$ is said to be self-dual. Further, if $\mathcal{B}$ is a von Neumann algebra, then $\mathcal{X}^\prime$ is also a Hilbert $\mathcal{B}$-module. Moreover, the $\mathcal{B}$-valued inner product on $\mathcal{X}$ extends to a $\mathcal{B}$-valued inner product on~$\mathcal{X}^\prime$.

\begin{theorem} [{\cite[Theorem 3.2]{Paschke}}] \label{thm;xprime}
Let $\mathcal{B}$ be a von Neumann algebra and $\mathcal{X}$ be a pre-Hilbert $\mathcal{B}$-module. Then the $\mathcal{B}$-valued inner product $\langle \cdot, \cdot \rangle$ on $\mathcal{X} \times \mathcal{X}$ extends to $\mathcal{X}^\prime \times \mathcal{X}^\prime$ in such a way as to make $\mathcal{X}^\prime$ into a self-dual Hilbert $\mathcal{B}$-module. Also, it satisfies $\langle \hat{x}, f \rangle = f(x)$ for $x \in \mathcal{X}$, $f \in \mathcal{X}^\prime$.
\end{theorem}

By using the map $x \mapsto \hat{x}$ for $x \in \mathcal{X}$, we have $\mathcal{X}$ as a submodule of $\mathcal{X}^\prime$. The following proposition extends an adjointable operator on $\mathcal{X}$ to an adjointable operator on $\mathcal{X}^\prime$.

\begin{proposition} [{\cite[Proposition 3.6, Corollary 3.7]{Paschke}}] \label{prop;pext}
Let $\mathcal{B}$ be a von Neumann algebra and $\mathcal{X}$ be a pre-Hilbert $\mathcal{B}$-module, then each $T \in \mathcal{P} ( \mathcal{X})$ extends to a unique $\widetilde{T} \in \mathcal{P} ( \mathcal{X}^\prime )$. The map $T \mapsto \widetilde{T}$ is a \(^*\)-isomorphism of $\mathcal{P} ( \mathcal{X} )$ into $\mathcal{P} ( \mathcal{X}^\prime )$. Moreover, $\widetilde{T}\hat{x} = \widehat{Tx}$ for all $x \in \mathcal{X}$.
\end{proposition}

%\noindent
%The last equality $\widetilde{T}\hat{x} = \widehat{Tx}$, for all $x \in \mathcal{X}$ can be observed from the proof of \cite[Proposition 3.6]{Paschke}.

Let $\mathcal{A}$ be a unital \(C^*\)-algebra, and let $\mathcal{B}$ be a von Neumann algebra. Suppose $\phi : \mathcal{A} \rightarrow \mathcal{B}$ is a completely positive map. By following Theorem \ref{thm;pd}, we get the Paschke dilation for $\phi$ given by the triple $\big (\mathcal{X}, \sigma, e \big )$. Then by using Proposition \ref{prop;pext}, we define a $\ast$-representation $\widetilde{\sigma} : \mathcal{A} \rightarrow \mathcal{P}( \mathcal{X}^\prime )$ by $\widetilde{\sigma}(a) := \widetilde{\sigma(a)}$. For $T \in \mathcal{P} ( \mathcal{X}^\prime)$, define a linear map $\phi_T : \mathcal{A} \rightarrow \mathcal{B}$ by $\phi_T(a) := \big \langle T \widetilde{\sigma}(a) \hat{e}, \hat{e} \big  \rangle$ for $a \in \mathcal{A}$. In \cite{Paschke}, the author has proved the following Radon--Nikodym type result in the setting of Hilbert modules for completely positive maps.

\begin{theorem} [{\cite[Proposition 5.4]{Paschke}}] \label{thm;prnd}
Let $\mathcal{A}$ be a unital \(C^*\)-algebra, and let $\mathcal{B}$ be a von Neumann algebra. Suppose $\phi : \mathcal{A} \rightarrow \mathcal{B}$ is a completely positive map with the Paschke dilation given by the triple $\big (\mathcal{X}, \sigma, e \big )$. Then there is an affine order isomorphism of $\big \{ T \in  \widetilde{\sigma}(\mathcal{A})^\prime \; \; : \; \; 0\leq T \leq 1_{\mathcal{X}^\prime} \}$ onto $[0, \phi] = \big \{ \theta \in CP(\mathcal{A}, \mathcal{B}) \; \; : \; \; \theta \leq \phi \big \}$ which is given by the map $T \mapsto \phi_T(\cdot) := \big \langle T \widetilde{\sigma}(\cdot) \hat{e}, \hat{e} \big  \rangle$. 
\end{theorem}

%%%%%%%%%%%%%%%%%%%%%%%%%%%%%%%%%%%%%%%%%%%%%%%%%%%%%%%%%%%%%%%%%%%%%%%%%%%%%%%%%%%%%%%%%

\section{Radon--Nikodym derivative} \label{sec; RND}

The notion of partial action was introduced by Exel in~\cite{Exel-1994-Circle-action-Cst-alg} and McClanahan in~\cite{McClanahan-1995-K-theory-Par-act}. In this section, we briefly recall the definition of partial action from Exel's book~\cite{E} and then we discuss Radon-Nikodym type result for a collection of UCP maps, which are invariant under partial action.

\begin{definition}[{\cite[definition 11.4]{E}}]
\label{def-partial-action}
Let $G$ be a group and $\mathcal{A}$ be a unital $C^\ast$-algebra. A partial action of the group $G$ on the $C^\ast$\nobreakdash-algebra $\mathcal{A}$ is a pair \( \tau := \big ( \{ \mathcal{A}_g \}_{g \in G},  \{ \tau_g \}_{g \in G}  \big )\),
consisting of a family of closed two-sided ideals $\mathcal{A}_g \subseteq \mathcal{A}$ and $\ast$-isomorphisms $\tau_g : \mathcal{A}_{g^{-1}} \rightarrow \mathcal{A}_g$ for all $g \in G$, satisfying the following
\begin{enumerate}
    \item  $\mathcal{A}_1 = \mathcal{A}$ and $\tau_1$ is the identity map on $\mathcal{A}_1 = \mathcal{A}$;
    \item $\tau_g \circ \tau_h \subseteq \tau_{gh}$ for all $g, h \in G$.
\end{enumerate}
\end{definition}
\noindent The domain of \(\tau_{g}\circ \tau_{h}\) is given by \(\{x\in \mathcal{A}_{h^{-1}}: \tau_h(x) \in \mathcal{A}_{g^{-1}}\} = \tau^{-1}_h(\mathcal{A}_{g^{-1}})\). The condition \(\tau_g\circ \tau_h \subseteq \tau_{gh}\) says that \(\tau_g(\tau_h(x)) = \tau_{gh}(x)\) for \(x\in \textup{dom}(\tau_g\circ \tau_h)\), that is, \(\tau_{gh}\) is an extension of \(\tau_g\circ \tau_h\). A partial action of \(G\) on a Hilbert space\(\mathcal{H}\) is a pair \( \big ( \{ \mathcal{K}_{g} \}_{g \in G} ; \{ U_{g} \}_{g \in G} \big )\), where \(\mathcal{K}_g\) is a closed subspace of \(\mathcal{H}\) and \(U_g\colon \mathcal{K}_{g^{-1}} \to \mathcal{K}_{g}\) is an unitary isomorphism and satisfying Conditions~(1) and~(2) of Definition~\ref{def-partial-action}. In a similar way, one can define the partial action of \(G\) on a Hilbert \(\mathcal{A}\)-module \(\mathcal{X}\).

For a unital $C^\ast$-algebra $\mathcal{A}$ and a Hilbert space $\mathcal{H}$, we recall the following notations:
\begin{align*}
\cba &= \big \{ \phi : \mathcal{A} \rightarrow \mathcal{B}(\mathcal{H}) \; \; : \; \;  \phi \; \; \text{is a completely bounded map} \big \}; \\
\cpa &= \big \{ \phi : \mathcal{A} \rightarrow \mathcal{B}(\mathcal{H}) \; \; : \; \;  \phi \; \; \text{is a completely positive map} \big \}; \\
\ucpa &= \big \{ \phi : \mathcal{A} \rightarrow \mathcal{B}(\mathcal{H}) \; \; : \; \;  \phi \; \; \text{is a unital completely positive map} \big \}.
\end{align*}
Let $\tau = \big ( \{ \mathcal{A}_g \}_{g \in G},  \{ \tau_g \}_{g \in G}  \big )$ be a partial action of \(G\) on a unital \(C^*\)-algebra \(\mathcal{A}\). Consider the following sets:
\begin{align*}
\cpapg &= \big \{ \phi \in \cpa  \; \; : \; \;  \phi(a) = \phi (\tau_g(a)) \; \; \text{for all} \; g \in G \; \; \text{and} \; \; a \in \mathcal{A}_{g^{-1}} \big \}; \\
\ucpapg &= \big \{ \phi \in \ucpa \; \; : \; \;  \phi(a) = \phi (\tau_g(a)) \; \; \text{for all} \; g \in G \; \; \text{and} \; \; a \in \mathcal{A}_{g^{-1}} \big \}.
\end{align*}
In this section, we prove the Radon--Nikodym type results in the subcollection $\ucpapg$ in the setting of Stinespring's and Paschke's dilation of a completely positive map.

Let $\tau = \big ( \{ \mathcal{A}_g \}_{g \in G},  \{ \tau_g \}_{g \in G}  \big )$ be a partial action of a group $G$ on a unital $C^\ast$-algebra $\mathcal{A}$. Suppose $\phi \in \ucpapg$ and  $\big (\pi, V, \mathcal{K} \big )$ is the minimal Stinespring representation for $\phi$ (see Theorem \ref{thm;sd}). Then for all $g \in G$ and $a \in \mathcal{A}_{g^{-1}}$, we have
\begin{equation*}
V^\ast\pi(a)V = \phi(a) = \phi\big (\tau_g(a) \big ) = V^\ast\pi \big (\tau_g(a) \big )V. 
\end{equation*}
Now for each fixed $g \in G$, define a Hilbert subspace of the Hilbert space $\mathcal{K}$ as
\begin{equation} \label{eqn; K g inverse}
\mathcal{K}_{g^{-1}} := \big [ \pi \big (\mathcal{A}_g \big )V \mathcal{H} \big ] = \big [ \pi \big (\tau_g(\mathcal{A}_{g^{-1}}) \big) V \mathcal{H} \big ].
\end{equation}
Next, define a map $U_g : \mathcal{K}_{g^{-1}} \rightarrow \mathcal{K}_{g}$ as 
\begin{equation} \label{eqn; U g}
U_g \left ( \sum\limits^n_i \pi (a_i) Vh_i \right ) = \sum\limits^n_i \pi\big (\tau_{g^{-1}}(a_i) \big ) Vh_i,  \; \; \; \; \text{where} \; \; a_i \in \mathcal{A}_g \; \; \text{and} \; \; h_i \in \mathcal{H}.
\end{equation}
Then one can see that $U_g$ is unitary. Since $g \in G$ was arbitrarily chosen, we get the unitary $U_g : \mathcal{K}_{g^{-1}} \rightarrow \mathcal{K}_{g}$ for each $g \in G$. Thus for a fixed $\phi \in \ucpapg$, we get a partial action of the group $G$ on the Hilbert space $\mathcal{K}$ which is given by the pair 
\begin{equation} \label{eqn; U phi}
U^{G_\tau}_\phi := \big ( \{ \mathcal{K}_{g} \}_{g \in G} ; \{ U_{g} \}_{g \in G} \big ).
\end{equation}
A result motivated by Proposition~\ref{prop;umsd} is given below:

\begin{proposition}
Let $G$ be a group and $\mathcal{A}$ be a unital $C^\ast$-algebra. Let $\tau = \big ( \{ \mathcal{A}_g \}_{g \in G},  \{ \tau_g \}_{g \in G}  \big )$ be a partial action of $G$ on $\mathcal{A}$. Consider a Hilbert space $\mathcal{H}$ and $\phi \in \ucpapg$ with the minimal Stinespring representation for $\phi$ is given by $\big (\pi, V, \mathcal{K} \big )$. Let $U^{G_\tau}_\phi = \big ( \{ \mathcal{K}_{g} \}_{g \in G} ; \{ U_{g} \}_{g \in G} \big )$ be the partial action of the group $G$ on the Hilbert space $\mathcal{K}$ obtained from $\phi$ (as described in Equation \eqref{eqn; U phi}). Then for all $g \in G$ and $a \in \mathcal{A}_g$,  we have 
\begin{equation*}
U_g\pi(a)U_g^\ast =\pi\big (\tau_{g^{-1}}(a) \big ) \big|_{\mathcal{K}_g}.
\end{equation*}
\end{proposition}
\begin{proof}
Let us recall that for each fixed $g\in G$, from Equation \eqref{eqn; K g inverse} we have the Hilbert space $\mathcal{K}_g := \big [ \pi \big (\mathcal{A}_{g^{-1}} \big )V \mathcal{H} \big ] = \big [ \pi \big (\tau_{g^{-1}}(\mathcal{A}_g) \big) V \mathcal{H} \big ].$ Let $b_i \in \mathcal{A}_{g^{-1}}$ and $h_i \in \mathcal{H}$ for $i = 1, 2,\cdots, n$. Then for $a \in \mathcal{A}_g$, by following Equation \eqref{eqn; U g}, we obtain
\begin{align*}
U_g\pi(a)U_g^\ast \left ( \sum\limits^n_i \pi (b_i) Vh_i \right ) &= U_g\pi(a) \sum\limits^n_i \pi\big (\tau_g(b_i) \big ) Vh_i \\
&= U_g \sum\limits^n_i \pi\big (a \tau_g(b_i) \big ) Vh_i \\
&= \sum\limits^n_i \pi\big ( \tau_{g^{-1}} \big (a \tau_g(b_i) \big ) \big ) Vh_i \\
&= \sum\limits^n_i \pi\big ( \tau_{g^{-1}}(a)\big ) \pi \big (b_i \big )  Vh_i = \pi\big ( \tau_{g^{-1}}(a)\big ) \left ( \sum\limits^n_i \pi (b_i) Vh_i \right ).
\end{align*}
This proves the result.
\end{proof}

The following lemma is the Radon--Nikodym type result in the subcollection $\ucpapg$ proved in the setting of Stinespring's dilation of a completely positive map.

\begin{lemma}\label{lem;rnd}
Let $G$ be a group and $\mathcal{A}$ be a unital $C^\ast$-algebra. Let $\tau = \big ( \{ \mathcal{A}_g \}_{g \in G},  \{ \tau_g \}_{g \in G}  \big )$ be a partial action of $G$ on $\mathcal{A}$. Consider a Hilbert space $\mathcal{H}$ and $\phi \in \ucpapg$ with the minimal Stinespring representation for $\phi$ is given by $\big (\pi, V, \mathcal{K} \big )$. Let $U^{G_\tau}_\phi = \big ( \{ \mathcal{K}_{g} \}_{g \in G} ; \{ U_{g} \}_{g \in G} \big )$ be the partial action of the group $G$ on the Hilbert space $\mathcal{K}$ obtained from $\phi$ (see Equation \eqref{eqn; U phi}). Then there exists an affine order isomorphism of the partially ordered convex set of operators 
\begin{equation*} 
\big \{ T \in  \pi(\mathcal{A})^\prime \; \; : \; \; 0 \leq T \leq 1_\mathcal{K}, \; \;  T(\mathcal{K}_g) \subseteq \mathcal{K}_g \textup{ and } TU_g = U_g T\big|_{\mathcal{K}_{g^{-1}}} \; \; \text{for all} \; \; g \in G \big  \}  
\end{equation*}
onto $[0, \phi] \cap \cpapg$ which is given by the map 
\begin{equation*}
T \mapsto \phi_T(\cdot) := V^* \pi (\cdot) T V.
\end{equation*}
\end{lemma}
\begin{proof}
Let $T \in \pi(\mathcal{A})^\prime$ be such that $0 \leq T \leq 1_\mathcal{K}$, $T(\mathcal{K}_g) \subseteq \mathcal{K}_g$ and  $TU_g = U_g T\big|_{\mathcal{K}_{g^{-1}}}$  for all $g \in G$. Then we show that the map $\phi_T : \mathcal{A} \rightarrow \mathcal{B}(\mathcal{H})$ defined by $\phi_T(a) := V^* \pi (a) T V$ for all $a \in \mathcal{A}$ is in $[0, \phi] \cap \cpapg$. Since $T \in \pi(\mathcal{A})^\prime$ and $0 \leq T \leq 1_\mathcal{K}$, by following Theorem \ref{thm;arnd}, we know that $\phi_T \in [0, \phi]$. Now we prove that $\phi_T \in \cpapg$. Let $g \in G$,  $a, b \in \mathcal{A}_{g^{-1}}$ and $h_1, h_2 \in \mathcal{H}$. Then we have
\begin{align*}
\big \langle \phi_T \big (\tau_g(b^\ast a) \big )h_1, h_2 \big \rangle &= \big \langle \phi_T \big (\tau_g(b)^\ast \tau_g(a) \big )h_1, h_2 \big \rangle  \\
&= \big \langle  T \pi \big (\tau_g(b)^\ast \tau_g(a) \big )Vh_1, Vh_2 \big \rangle \\
&= \big \langle T \pi \big (\tau_g(a) \big)Vh_1,   \pi \big (\tau_g(b) \big)Vh_2 \big \rangle \\
&= \big \langle T U_g^\ast \big (\pi(a)Vh_1 \big ),  U_g^\ast \big (\pi(b)Vh_2 \big ) \big \rangle \\
&= \big \langle U_g T U_g^\ast \big (\pi(a)Vh_1 \big ),  \pi(b)Vh_2 \big \rangle \\
&= \big \langle \pi(b)^\ast U_g T U_g^\ast \big (\pi(a)Vh_1 \big ), Vh_2  \big \rangle \\
&= \big \langle \pi(b)^\ast U_g U_g^\ast T \big ( \pi(a) Vh_1 \big ), Vh_2  \big \rangle \\
&= \big \langle \pi(b)^\ast T \pi(a) Vh_1, Vh_2  \big \rangle \\
&= \big \langle V^\ast T \pi(b^\ast a) Vh_1, h_2  \big \rangle \\
&= \big \langle \phi_T(b^\ast a) h_1, h_2\big \rangle.
\end{align*}
Since $g \in G$, $a, b \in \mathcal{A}_{g^{-1}}$ and $h_1, h_2 \in \mathcal{H}$ were arbitrarily chosen, we obtain  $\phi_T \in \cpapg$.

Conversely, assume that $\phi_T \in [0, \phi] \cap \cpapg$. Since $\phi_T \in [0, \phi]$ by following Theorem \ref{thm;arnd}, we get $T \in \pi(\mathcal{A})^\prime$ with the property that $0 \leq T \leq 1_\mathcal{K}$. We also have  $\phi_T \in \cpapg$.  Thus for $g \in G$,  $a, b \in \mathcal{A}_{g^{-1}}$ and $h_1, h_2 \in \mathcal{H}$, we have
\begin{align*}
\big \langle U_g^\ast T U_g \pi \big (\tau_g(a) \big )Vh_1, \pi  \big(\tau_g(b)  \big )Vh_2 \big \rangle
&= \big \langle T U_g \pi  \big (\tau_g(a)  \big )Vh_1, U_g \pi  \big (\tau_g(b)  \big)Vh_2  \big \rangle \\
&= \big \langle T  \pi(a)Vh_1,  \pi(b)Vh_2  \big \rangle \\
&= \big \langle V^\ast\pi(b)^\ast T \pi(a) Vh_1, h_2  \big \rangle \\
&= \big \langle V^\ast T \pi(b^\ast a) Vh_1, h_2  \big \rangle \\
&= \big \langle \phi_T(b^\ast a) h_1, h_2  \big \rangle \\
&= \big \langle \phi_T \big (  \tau_g(b^\ast a) \big ) h_1, h_2  \big \rangle \\
&= \big \langle T \pi \big  ( \tau_g(b)^\ast \tau_g(a) \big ) Vh_1, Vh_2  \big \rangle \\
&= \big \langle T \pi \big ( \tau_g(a) \big ) Vh_1, \pi \big  (\tau_g(b) \big ) Vh_2 \big \rangle. 
\end{align*}
From Equation \eqref{eqn; K g inverse} we know that $\mathcal{K}_{g^{-1}} := \big [ \pi \big (\mathcal{A}_g \big )V \mathcal{H} \big ] = \big [ \pi \big (\tau_g(\mathcal{A}_{g^{-1}}) \big) V \mathcal{H} \big ]$. Since  $a, b \in \mathcal{A}_{g^{-1}}$ and $h_1, h_2 \in \mathcal{H}$ were arbitrarily chosen, we get $U_g^\ast T U_g = T\big|_{\mathcal{K}_{g^{-1}}}$ and thus $T U_g = U_g T\big|_{\mathcal{K}_{g^{-1}}}$. As $g \in G$ was arbitrary, we get the result.
\end{proof}

Now, in the remaining part of the section, we see similar results in the setting of Paschke's dilation of a completely positive map. For this, first, we set some notations to state these results. Let $G$ be a group and $\mathcal{A}$ be a unital $C^\ast$-algebra. Let $\tau = \big ( \{ \mathcal{A}_g \}_{g \in G},  \{ \tau_g \}_{g \in G}  \big )$ be a partial action of $G$ on $\mathcal{A}$. Consider a Hilbert space $\mathcal{H}$ and $\phi \in \ucpapg$. Then by applying Theorem \ref{thm;pd}, we get the triple $\big ( \mathcal{X}, \sigma, e \big )$, where $\mathcal{X}$ is a right Hilbert-$\mathcal{B}(\mathcal{H})$ module, $e \in \mathcal{X}$, and $\sigma : \mathcal{A} \rightarrow \mathcal{P} \big ( \mathcal{X} \big)$ a *-representation such that 
\begin{equation*}
\phi(a) = \big \langle \sigma(a)e, e  \big \rangle \; \; \; \text{for all} \; \; a \in \mathcal{A}.
\end{equation*}
We recall that the set $\big \{ \sigma(a)(eS) \;  : \;  a \in \mathcal{A}, S \in \mathcal{B}(\mathcal{H}) \big \}$ spans a dense subspace of $\mathcal{X}$ (see Theorem~\ref{thm;pd}).  Since $\phi \in \ucpapg$, for $g \in G$ and $a \in \mathcal{A}_{g^{-1}}$, we have 
\begin{equation*}
\big \langle \sigma(a)e, e \big \rangle = \phi(a)  = \phi \big(\tau_g(a) \big) = \big \langle \sigma(\tau_g(a))e, e \big \rangle.
\end{equation*}
Now for each fixed $g \in G$, define a Hilbert-$\mathcal{B}(\mathcal{H})$ submodule of the Hilbert-$\mathcal{B}(\mathcal{H})$ module $\mathcal{X}$ as
\begin{equation} \label{eqn; X g inverse}
\mathcal{X}_{g^{-1}} := \big [ \sigma \big (\mathcal{A}_g \big ) \big ( e \mathcal{B}(\mathcal{H})  \big ) \big ] = \big [ \sigma \big (\tau_g(\mathcal{A}_{g^{-1}}) \big) \big ( e \mathcal{B}(\mathcal{H})  \big ) \big ].
\end{equation}
From Theorem \ref{thm;pd}, it is clear that $\mathcal{X}_{g^{-1}} \subseteq \mathcal{X}$ for all $g \in G$.
Next, define a map $W_g : \mathcal{X}_{g^{-1}} \rightarrow \mathcal{X}_{g}$ as 
\begin{equation} \label{eqn; W g}
W_g \left ( \sum\limits^n_i \sigma (a_i) eS_i \right ) = \sum\limits^n_i \sigma\big (\tau_{g^{-1}}(a_i) \big ) eS_i,  \; \; \; \; \text{where} \; \; a_i \in \mathcal{A}_g  \; \; \text{and} \; \; S_i \in \mathcal{B}(\mathcal{H}). 
\end{equation}
Then one can see that $W_g$ is a Hilbert module isomorphism. Since $g \in G$ was arbitrarily chosen, we get the Hilbert module isomorphism $W_g : \mathcal{X}_{g^{-1}} \rightarrow \mathcal{X}_{g}$ for each $g \in G$. Thus for a fixed $\phi \in \ucpapg$, we get a partial action of the group $G$ on the Hilbert-$\mathcal{B}(\mathcal{H})$ module $\mathcal{X}$, which is given by the pair 
\begin{equation} \label{eqn; W phi}
W^{G_\tau}_\phi := \big ( \{ \mathcal{X}_{g} \}_{g \in G} ; \{ W_{g} \}_{g \in G} \big ).
\end{equation}
A result motivated by Theorem \ref{thm;umpd} is given below:

\begin{proposition}
Let $G$ be a group and $\mathcal{A}$ be a unital $C^\ast$-algebra. Let $\tau = \big ( \{ \mathcal{A}_g \}_{g \in G},  \{ \tau_g \}_{g \in G}  \big )$ be a partial action of $G$ on $\mathcal{A}$. Consider a Hilbert space $\mathcal{H}$ and $\phi \in \ucpapg$ with the Paschke's dilation for $\phi$ is given by $\big ( \mathcal{X}, \sigma, e \big )$. Let $W^{G_\tau}_\phi = \big ( \{ \mathcal{X}_{g} \}_{g \in G} ; \{ W_{g} \}_{g \in G} \big )$ be the partial action of the group $G$ on the Hilbert-$\mathcal{B}(\mathcal{H})$ module $\mathcal{X}$ obtained from $\phi$ (as described in Equation \eqref{eqn; W phi}). Then for all $g \in G$ and $a \in \mathcal{A}_g$,  we have 
\begin{equation*}
W_g\sigma(a)W_g^\ast =\sigma\big (\tau_{g^{-1}}(a) \big ) \big|_{\mathcal{X}_g}.
\end{equation*}
\end{proposition}
\begin{proof}
Let us recall that for each fixed $g\in G$, from Equation \eqref{eqn; X g inverse} we have the Hilbert-$\mathcal{B}(\mathcal{H})$ submodule given by $\mathcal{X}_g := \big [ \sigma \big (\mathcal{A}_{g^{-1}} \big ) \big ( e \mathcal{B}(\mathcal{H}) \big ) \big ] = \big [ \sigma \big (\tau_{g^{-1}}(\mathcal{A}_g) \big) ( e \mathcal{B}(\mathcal{H}) \big ) \big ].$ Let $b_i \in \mathcal{A}_{g^{-1}}$ and $S_i \in \mathcal{B}(\mathcal{H})$ for $i = 1, 2,\cdots, n$. Then for $a \in \mathcal{A}_g$, by following Equation \eqref{eqn; W g}, we obtain
\begin{align*}
W_g\sigma(a)W_g^\ast \left ( \sum\limits^n_i \sigma (b_i) eS_i \right ) &= W_g\sigma(a) \sum\limits^n_i \sigma\big (\tau_g(b_i) \big ) eS_i \\
&= W_g \sum\limits^n_i \sigma\big (a \tau_g(b_i) \big ) eS_i \\
&= \sum\limits^n_i \sigma\big ( \tau_{g^{-1}} \big (a \tau_g(b_i) \big ) \big ) eS_i \\
&= \sum\limits^n_i \sigma\big ( \tau_{g^{-1}}(a)\big ) \sigma \big (b_i \big )  eS_i = \sigma\big ( \tau_{g^{-1}}(a)\big ) \left ( \sum\limits^n_i \sigma (b_i) eS_i \right ).
\end{align*}
This proves the result.
\end{proof}

For each fixed $g \in G$, we have the Hilbert-$\mathcal{B}(\mathcal{H})$ submodule $\mathcal{X}_g$ of the Hilbert-$\mathcal{B}(\mathcal{H})$ module $\mathcal{X}$ as defined in Equation \eqref{eqn; X g inverse}. Now consider the set 
\begin{equation*}
\mathcal{X}_g^\prime := \big \{ f : \mathcal{X}_g \rightarrow \mathcal{B}(\mathcal{H}) \; \; : \; \; f \; \; \text{is a bounded $\mathcal{B}(\mathcal{H})$-module map} \big \}.
\end{equation*} 
For a fixed $x \in \mathcal{X}_g$, we define a map $\hat{x} : \mathcal{X}_g \rightarrow \mathcal{B}(\mathcal{H})$ as $\hat{x}(y) := \langle y, x \rangle$ for $y \in \mathcal{X}_g$. Thus, with this identification, we have $\mathcal{X}_g \subseteq \mathcal{X}_g^\prime$ for all $g \in G$. Also, from Equation \eqref{eqn; X g inverse} and Equation \eqref{eq; X prime}, we know that $\mathcal{X}_g \subseteq \mathcal{X} \subseteq \mathcal{X}^\prime$ for all $g \in G$. So, in conclusion, we obtain
\begin{equation} \label{eq; containment of Xg and Xg prime}
\mathcal{X}_g \subseteq \mathcal{X} \subseteq \mathcal{X}^\prime  \; \; \; \text{and}  \; \; \;  \mathcal{X}_g \subseteq \mathcal{X}^\prime_g  \; \; \; \text{for all} \; \; g \in G.
\end{equation}
Since $\mathcal{B}(\mathcal{H})$ is a von Neumann algebra by following Theorem \ref{thm;xprime}, we see that $\mathcal{X}^\prime_g$ is a self dual Hilbert-$\mathcal{B}(\mathcal{H})$ module (with the $\mathcal{B}(\mathcal{H})$-valued inner product on $\mathcal{X}_g$ extends to a $\mathcal{B}(\mathcal{H})$-valued inner product on $\mathcal{X}_g^\prime$). Also, it satisfies $\langle \hat{x}, f \rangle = f(x)$ for $x \in \mathcal{X}_g$, $f \in \mathcal{X}_g^\prime$. Further, by using Proposition \ref{prop;pext}, the map $W_g : \mathcal{X}_{g^{-1}} \rightarrow \mathcal{X}_{g}$ (defined in Equation \eqref{eqn; W g}) can be uniquely extended to the map $\widetilde{W_g} : \mathcal{X}^\prime_{g^{-1}} \rightarrow \mathcal{X}^\prime_{g}$ for all $g \in G$. For $g \in G$ and $a \in \mathcal{A}_g$, by following Equation \eqref{eqn; X g inverse}  we get $\sigma(a)e \in \mathcal{X}_{g^{-1}}$ and then  definition of $\widetilde{\sigma}$ and finally Equation \eqref{eq; containment of Xg and Xg prime}  imply that  \begin{equation*}
\widetilde{\sigma}(a) \hat{e} = \widetilde{\sigma(a)} \hat{e} = \widehat{\sigma(a)e}  \in \mathcal{X}^\prime_{g^{-1}}.
\end{equation*}
For $g \in G$,  define 
\begin{equation}  \label{eqn; Y g inverse}
\mathcal{Y}_{g^{-1}} := \big \{ \widetilde{\sigma}(a) \hat{e} \; : \; a \in \mathcal{A}_g  \big \}  \subseteq \mathcal{X}^\prime_{g^{-1}}.
\end{equation}
Then we get $\widetilde{W_{g}} \big ( \widetilde{\sigma}(a) \hat{e} \big ) =  \widetilde{\sigma} \big( \tau_{g^{-1}}(a) \big ) \hat{e} $ for all $\widetilde{\sigma}(a) \hat{e} \in \mathcal{Y}_{g^{-1}}$. Thus $\widetilde{W_{g}} \big ( \mathcal{Y}_{g^{-1}} \big ) = \mathcal{Y}_g$. For a fixed $\phi \in \ucpapg$, we consider the following pair 
\begin{equation} \label{eqn; W phi tilde}
\widetilde{W^{G_\tau}_\phi} := \big ( \{ \mathcal{X}^\prime_{g} \}_{g \in G} ; \{ \widetilde{W_{g}} \}_{g \in G} \big ).
\end{equation}
The pair given by \(\widetilde{W^{G_\tau}_\phi}\) is a partial action of \(G\) on the Hilbert \(\mathcal{B}(\mathcal{H})\)-module \(\mathcal{X}^{\prime}\).
%\textcolor{red}{Whether this pair is a partial action of $G$ on some Hilbert-$\mathcal{B}(\mathcal{H})$ module? If yes, then on which module? } 

The following lemma is the Radon--Nikodym type result in the subcollection $\ucpapg$ proved in the setting of Paschke's dilation of a completely positive map.

\begin{lemma}\label{lem;prnd}
Let $G$ be a group and $\mathcal{A}$ be a unital $C^\ast$-algebra. Let $\tau = \big ( \{ \mathcal{A}_g \}_{g \in G},  \{ \tau_g \}_{g \in G}  \big )$ be a partial action of $G$ on $\mathcal{A}$. Consider a Hilbert space $\mathcal{H}$ and $\phi \in \ucpapg$ with the Paschke's dilation for $\phi$ is given by $\big ( \mathcal{X}, \sigma, e \big )$. Let $\widetilde{W^{G_\tau}_\phi} := \big ( \{ \mathcal{X}^\prime_{g} \}_{g \in G} ; \{ \widetilde{W_{g}} \}_{g \in G} \big )$ be the pair (as described in Equation \eqref{eqn; W phi tilde}).  Then there exists an affine order isomorphism of the partially ordered convex set of operators 
\begin{equation*}
\big \{ T \in  \widetilde{\sigma}(\mathcal{A})^\prime \; \; : \; \; 0 \leq T \leq 1_{\mathcal{X}^\prime}, \; \;  T(\mathcal{Y}_g) \subseteq \mathcal{Y}_g \textup{ and } T\widetilde{W}_g \big |_{\mathcal{Y}_{g^{-1}}} = \widetilde{W}_g T \big |_{\mathcal{Y}_{g^{-1}}}  \; \; \text{for all} \; \; g \in G \big  \}  
\end{equation*}
onto $[0, \phi] \cap \cpapg$ which is given by the map 
\begin{equation*}
T \mapsto \phi_T(\cdot) := \big \langle T \widetilde{\sigma}(\cdot) \hat{e}, \hat{e} \big  \rangle.
\end{equation*}
\end{lemma}
\begin{proof}
Let $T \in  \widetilde{\sigma}(\mathcal{A})^\prime$ be such that $0 \leq T \leq 1_{\mathcal{X}^\prime}$, $T(\mathcal{Y}_g) \subseteq \mathcal{Y}_g$ and  $T\widetilde{W}_g \big |_{\mathcal{Y}_{g^{-1}}} = \widetilde{W}_g T \big |_{\mathcal{Y}_{g^{-1}}}$  for all $g \in G$. Then we show that the map $\phi_T : \mathcal{A} \rightarrow \mathcal{B}(\mathcal{H})$ defined by $\phi_T(a) := \big \langle T \widetilde{\sigma}(a) \hat{e}, \hat{e} \big  \rangle$ for all $a \in \mathcal{A}$ is in $[0, \phi] \cap \cpapg$. Since $T \in  \widetilde{\sigma}(\mathcal{A})^\prime$ and $0 \leq T \leq 1_{\mathcal{X}^\prime}$, by following Theorem \ref{thm;prnd}, we know that $\phi_T \in [0, \phi]$. Now we prove that $\phi_T \in \cpapg$. Let $g \in G$ and $a, b \in \mathcal{A}_g$. Then we have
\begin{align*}
\phi_T \big (\tau_{g^{-1}}(b^\ast a) \big ) &=  \big \langle T \widetilde{\sigma}\big (\tau_{g^{-1}}(b^\ast a) \big ) \hat{e}, \hat{e} \big  \rangle \\
&= \big \langle T \widetilde{\sigma} \big (\tau_{g^{-1}}(a) \big )  \hat{e}, \widetilde{\sigma} \big (\tau_{g^{-1}}(b) \big ) \hat{e} \big  \rangle \\
&= \big \langle T \widetilde{W_{g}} \big ( \widetilde{\sigma}(a) \hat{e} \big ), \widetilde{W_{g}} \big ( \widetilde{\sigma}(b) \hat{e} \big ) \big  \rangle \\
&= \big \langle \widetilde{W_{g}}^\ast T \widetilde{W_{g}} \big ( \widetilde{\sigma}(a) \hat{e} \big ),  \widetilde{\sigma}(b) \hat{e} \big   \rangle\\ 
&= \big \langle \widetilde{W_{g}}^\ast \widetilde{W_{g}} T  \big ( \widetilde{\sigma}(a) \hat{e} \big ),  \widetilde{\sigma}(b) \hat{e} \big   \rangle\\
&= \big \langle  T  \big ( \widetilde{\sigma}(a) \hat{e} \big ),  \widetilde{\sigma}(b) \hat{e} \big   \rangle \\
&= \big \langle  T  \big ( \widetilde{\sigma}(b^\ast a) \hat{e} \big ),   \hat{e} \big   \rangle \\
&= \phi_T \big (b^\ast a \big).
\end{align*}
Since $g \in G$ and $a, b \in \mathcal{A}_g$ were arbitrarily chosen, we get $\phi_T \in \cpapg$.

Conversely, assume that $\phi_T \in [0, \phi] \cap \cpapg$. Since $\phi_T \in [0, \phi]$ by following Theorem \ref{thm;prnd}, we get $T \in  \widetilde{\sigma}(\mathcal{A})^\prime$ with the property that $0 \leq T \leq 1_{\mathcal{X}^\prime}$. We also know that  $\phi_T \in \cpapg$.  Thus for $g \in G$ and  $a, b \in \mathcal{A}_g$, we have
\begin{align*}
\big \langle \widetilde{W_{g}}^\ast T \widetilde{W_{g}} \big (\widetilde{\sigma}(a) \hat{e} \big ), \widetilde{\sigma}(b) \hat{e} \big \rangle &= \big \langle  T \widetilde{W_{g}} \big ( \widetilde{\sigma}(a) \hat{e} \big ), \widetilde{W_{g}} \big ( \widetilde{\sigma}(b) \hat{e} \big ) \big \rangle \\
&= \big \langle  T \widetilde{\sigma} \big( \tau_{g^{-1}}(a) \hat{e} \big ), \widetilde{\sigma} \big( \tau_{g^{-1}}(b) \hat{e} \big ) \big \rangle \\
&= \big \langle \widetilde{\sigma} \big( \tau_{g^{-1}}(b) \big )^\ast T \widetilde{\sigma} \big( \tau_{g^{-1}}(a) \hat{e} \big ),  \hat{e} \big \rangle \\
&= \big \langle T \widetilde{\sigma} \big( \tau_{g^{-1}}(b^\ast a) \hat{e} \big ),  \hat{e} \big \rangle \\
&=  \phi_T \big (\tau_{g^{-1}}(b^\ast a) \big ) \\
&= \phi_T \big (b^\ast a \big) \\
&= \big \langle T \widetilde{\sigma} \big (b^\ast a \big ) \hat{e},  \hat{e} \big \rangle \\
&= \big \langle T  \big (\widetilde{\sigma}(a) \hat{e} \big ), \widetilde{\sigma}(b) \hat{e} \big \rangle. 
\end{align*}
For all $g \in G$, we have $\mathcal{Y}_{g^{-1}} := \big \{ \widetilde{\sigma}(a) \hat{e} \; : \; a \in \mathcal{A}_g  \big \}$. Since  $a, b \in \mathcal{A}_g$ were arbitrarily chosen, we get $\widetilde{W_{g}}^\ast T \widetilde{W_{g}} \big|_{\mathcal{Y}_{g^{-1}}} = T\big|_{\mathcal{Y}_{g^{-1}}}$ and thus $T\widetilde{W}_g \big |_{\mathcal{Y}_{g^{-1}}} = \widetilde{W}_g T \big |_{\mathcal{Y}_{g^{-1}}}$. As $g \in G$ was arbitrary, we get the result.
\end{proof}

%%%%%%%%%%%%%%%%%%%%%%%%%%%%%%%%%%%%%%%%%%%%%%%%%%%%%%%%%%%%%%%%%%%%%%%%%%%%%%%%%%%%%%%%%

\section{Main Result} \label{sec; main result}

Let $G$ be a group and $\mathcal{A}$ be a unital $C^\ast$-algebra. Suppose $\tau = \big ( \{ \mathcal{A}_g \}_{g \in G},  \{ \tau_g \}_{g \in G}  \big )$ is a  partial action of $G$ on the $C^\ast$-algebra $\mathcal{A}$. For a Hilbert space $\mathcal{H}$, recall the definitions of \(\cba\), \(\cpa\), \( \ucpa\), \(\cpapg\) and \(\ucpapg\) from Section~\ref{sec; RND}. 
%\begin{align*}
%\cba &= \big \{ \phi : \mathcal{A} \rightarrow \mathcal{B}(\mathcal{H}) \; \; : \; \;  \phi \; \; \text{is a completely bounded map} \big \}; \\
%\cpa &= \big \{ \phi : \mathcal{A} \rightarrow \mathcal{B}(\mathcal{H}) \; \; : \; \;  \phi \; \; \text{is a completely positive map} \big \}; \\
%\ucpa &= \big \{ \phi : \mathcal{A} \rightarrow \mathcal{B}(\mathcal{H}) \; \; : \; \;  \phi \; \; \text{is a unital completely positive map} \big \}; \\
%\cpapg &= \big \{ \phi \in \cpa  \; \; : \; \;  \phi(a) = \phi (\tau_g(a)) \; \; \text{for all} \; g \in G \; \; \text{and} \; \; a \in \mathcal{A}_{g^{-1}} \big \}; \\
%\ucpapg &= \big \{ \phi \in \ucpa \; \; : \; \;  \phi(a) = \phi (\tau_g(a)) \; \; \text{for all} \; g \in G \; \; \text{and} \; \; a \in \mathcal{A}_{g^{-1}} \big \}.
%\end{align*}
Let $\phi_1, \phi_2 \in \ucpapg$ and $\lambda \in (0, 1)$. Then for all $g \in G$ and $a \in \mathcal{A}_{g^{-1}}$, one may note that 
\begin{align*}
\big ( \lambda \phi_1 + (1-\lambda) \phi_2 \big ) (a) = \lambda \phi_1(a) + (1-\lambda) \phi_2 (a) &= \lambda \phi_1 \big ( \tau_g(a) \big ) + (1-\lambda) \phi_2 \big ( \tau_g(a) \big ) \\
&=  \big ( \lambda \phi_1 + (1-\lambda) \phi_2 \big ) \big ( \tau_g(a) \big ).
\end{align*}
This shows that $\ucpapg$ is a convex subset of the locally convex topological vector space $\cba$. Now let $\big \{ \phi_\alpha  \big \}_{\alpha}$ be a net of elements in $\ucpapg$ such that $\phi_\alpha \rightarrow \phi \in \ucpa$ in the BW-topology (here, $\phi \in \ucpa$ becuase $\ucpa$ is a compact set). That is, $\phi_\alpha(a) \rightarrow \phi (a)$ in the weak operator topology in $\mathcal{B}(\mathcal{H})$ for all $a \in \mathcal{A}$. Let $g \in G$ and $a \in \mathcal{A}_{g^{-1}}$. Then $\phi_\alpha\big ( \tau_g(a) \big ) \rightarrow \phi \big ( \tau_g(a) \big )$ and $\phi_\alpha(a) \rightarrow \phi (a)$. However, $\phi_\alpha\big ( \tau_g(a) \big ) = \phi_\alpha(a)$ for each $\alpha$, and hence $\phi \big ( \tau_g(a) \big ) = \phi (a)$. Since $g \in G$ and $a \in \mathcal{A}_{g^{-1}}$ were chosen arbitrarily, we get that the convex set $\ucpapg$ is also compact in the BW-topology.

Now in this section, we state and prove the main result of this article (Theorem~\ref{thm;egiucp}). The result characterizes extreme points of the compact and convex set $\ucpapg$ using the techniques of Stinespring's and Paschke's dilation of completely positive maps. We denote the collection of all extreme points of the set $\ucpapg$ by the notation
\begin{equation*}
\text{ext} \big (\ucpapg \big ).
\end{equation*}
Before stating the main result, we recall a definition from \cite{Arveson1}.

\begin{definition} \label{def; faithful subspace}
Let $\mathcal{H}$ be a Hilbert space and $\mathcal{H}_0$ be a closed subspace of $\mathcal{H}$. Let $P$ denote the projection of $\mathcal{H}$ onto $\mathcal{H}_0$. We say that $\mathcal{H}_0$ is a faithful subspace of $\mathcal{H}$ for a set $\mathcal{M} \subseteq \mathcal{B}(\mathcal{H})$ if $PT\big |_{\mathcal{H}_0} = 0$ for $T \in \mathcal{M}$, then $T = 0$. 
\end{definition}

\noindent
We consider the following two sets 
\begin{equation} \label{eqn; cs phi}
\mathcal{C}^S_\phi := \big \{ \alpha_1 T_1 + \alpha_2 T_2 \; \; : \; \; \alpha_1, \alpha_2 \in [-1, 1], \; \; T_1, \; T_2 \; \text{as in Lemma \ref{lem;rnd}} \big \};  
\end{equation}
\begin{equation} \label{eqn; cp phi}
\mathcal{C}^P_\phi := \big \{ \alpha_1 T_1 + \alpha_2 T_2 \; \; : \; \; \alpha_1, \alpha_2 \in [-1, 1], \; \; T_1, \; T_2 \; \text{as in Lemma \ref{lem;prnd}} \big \}.   
\end{equation}
We direct the reader mainly to Lemmas~\ref{lem;rnd}, \ref{lem;prnd} and Equations~\eqref{eqn; K g inverse}, \eqref{eqn; U g}, \eqref{eqn; Y g inverse}, \eqref{eqn; W phi tilde} to follow the notations used in the theorem below. 

\begin{theorem}\label{thm;egiucp}
Let $G$ be a group and $\mathcal{A}$ be a unital $C^\ast$-algebra. Let $\tau = \big ( \{ \mathcal{A}_g \}_{g \in G},  \{ \tau_g \}_{g \in G}  \big )$ be a partial action of $G$ on $\mathcal{A}$. Consider a Hilbert space $\mathcal{H}$ and $\phi \in \ucpapg$ with the minimal Stinespring representation for $\phi$ is given by $\big (\pi, V, \mathcal{K} \big )$ and the Paschke's dilation for $\phi$ is given by $\big ( \mathcal{X}, \sigma, e \big )$. Then the following statements are equivalent:
\begin{enumerate}
\item the UCP map  $\phi \in \text{ext} \big (\ucpapg \big )$;
\item if $\phi_0 \in [0, \phi] \cap t\ucpapg$, then $\phi_0 = t\phi$, where $0 < t < 1$;
\item the subspace $[V\mathcal{H}]$ is a faithful subspace of the Hilbert space $\mathcal{K}$ for the set $\mathcal{C}^S_\phi$;
%\begin{equation*}
%\big \{ T \in  \pi(\mathcal{A})^\prime \; \; : \; \; 0 \leq T \leq 1_\mathcal{K}, \; \;  T(\mathcal{K}_g) \subseteq \mathcal{K}_g \; \; TU_g = U_g T\big|_{\mathcal{K}_{g^{-1}}} \; \; \text{for all} \; \; g \in G \big  \};
%\end{equation*}
\item if any $T \in \mathcal{C}^P_\phi$  
%\begin{equation*}
%\big \{ T \in  \widetilde{\sigma}(\mathcal{A})^\prime \; \; : \; \; 0 \leq T \leq 1_\mathcal{X}^\prime, \; \;  T(\mathcal{Y}_g) \subseteq \mathcal{Y}_g \; \; T\widetilde{W}_g \big |_{\mathcal{Y}_{g^{-1}}} = \widetilde{W}_g T \big |_{\mathcal{Y}_{g^{-1}}}  \; \; \text{for all} \; \; g \in G \big  \}  
%\end{equation*} 
satisfies $\langle T \hat{e}, \hat{e} \rangle = 0$, then $T=0$. 
\end{enumerate}
\end{theorem}
\begin{proof}
First, we prove that (1) $\iff$ (2). Let us consider $\phi \in \text{ext} \big (\ucpapg \big )$ and $\phi_0 \in [0, \phi] \cap t\ucpapg$, where $0 < t < 1$. Thus there exists 
$\phi_1 \in \ucpapg$ such that  $\phi_0 = t \phi_1 \leq \phi$. Since $\phi_0 \in [0, \phi]$, using Theorem \ref{thm;arnd} we obtain 
\begin{equation} \label{eqn; t phi 1}
\phi_0(\cdot) = t\phi_1(\cdot) = V^* T \pi(\cdot) V
\end{equation}
for some $T \in \pacom$ with $0 \leq T \leq  1_\mathcal{K}$. Observe that $0 < t < 1$ and $\phi_0 \in [0, \phi] \cap t\ucpapg$. This implies that $\phi_0 \in  [0, \phi] \cap \cpapg$. Hence by following Lemma~\ref{lem;rnd}, we get that the operator $T$ satisfies the properties that $T(\mathcal{K}_g) \subseteq \mathcal{K}_g$ and $TU_g = U_g T\big|_{\mathcal{K}_{g^{-1}}}$ for all $g \in G$. Then one may observe that $1_\mathcal{K} - T \in \pacom$ with $0 \leq 1_\mathcal{K} - T \leq  1_\mathcal{K}$ and $\big (  1_\mathcal{K} - T \big ) \big (\mathcal{K}_g \big ) \subseteq \mathcal{K}_g$ and $\big (  1_\mathcal{K} - T \big )U_g = U_g \big ( 1_\mathcal{K} - T \big )\big|_{\mathcal{K}_{g^{-1}}}$ for all $g \in G$. Thus Lemma \ref{lem;rnd} implies that $V^\ast \big (  1_\mathcal{K} - T \big ) \pi V \in [0, \phi] \cap \cpapg$. Define 
\begin{equation} \label{eqn;  phi 2}
\phi_2(\cdot) := V^\ast \frac{(1_\mathcal{K} - T)}{1 - t} \pi(\cdot) V.
\end{equation}
Then
\begin{equation*}
\phi_2 \big (1_\mathcal{A} \big ) = V^\ast  \pi \big (1_\mathcal{A} \big ) \frac{(1_\mathcal{K} - T)}{1 - t} V  = \frac{1}{1 - t} (V^\ast 1_\mathcal{K}   V - V^\ast T V) = \frac{(1 - t)1_\mathcal{H}}{1 - t} = 1_\mathcal{H}.
\end{equation*} 
This shows $\phi_2 \in \ucpapg$. Now for $\phi_1, \phi_2 \in \ucpapg$ by following Equations~\eqref{eqn; t phi 1} and \eqref{eqn;  phi 2}, we have
\begin{equation*}
 \phi(\cdot) = V^\ast \pi(\cdot) V = V^\ast T \pi(\cdot) V + (1-t) V^\ast \frac{(1_\mathcal{K} - T)}{1 - t} \pi(\cdot) V = t\phi_1(\cdot) + (1-t)\phi_2(\cdot).
\end{equation*}
Since $\phi \in \text{ext} \big (\ucpapg \big )$, we get $\phi = \phi_1 = \phi_2$. Hence, $\phi_0 = t \phi_1 = t \phi$.

Conversely, we assume that if $\phi_0 \in [0, \phi] \cap t\ucpapg$, then $\phi_0 = t\phi$, where $0 < t < 1$. Suppose $\phi = \lambda\phi_1 + (1 - \lambda) \phi_2$, for some $\phi_1, \phi_2 \in \ucpapg$ and $0 < \lambda < 1$. Then $\lambda\phi_1 \in [0, \phi] \cap \lambda\ucpapg$. Hence from the assumption, we get $\lambda\phi_1 = \lambda\phi$, and similarly $(1 - \lambda) \phi_2 = (1- \lambda) \phi$. This implies $\phi_1 = \phi = \phi_2$.

Now we prove that (1) $\iff$ (3). First, we assume that the subspace $[V\mathcal{H}]$ is a faithful subspace of the Hilbert space $\mathcal{K}$ for the set $\mathcal{C}^S_\phi$.
Let $\phi_1, \phi_2 \in \ucpapg$, and $0 < \lambda < 1$ be such that $\phi = \lambda \phi_1 + (1- \lambda) \phi_2$. As $\lambda \phi_1 \leq \phi$ and $\lambda \phi_1 \in \cpapg$ by following Lemma \ref{lem;rnd} we get $\lambda \phi_1 = V^* T \pi V$ for some $T \in  \pi(\mathcal{A})^\prime$ such that $0 \leq T \leq 1_\mathcal{K}$ satisfying the properties $T(\mathcal{K}_g) \subseteq \mathcal{K}_g$ and $TU_g = U_g T\big|_{\mathcal{K}_{g^{-1}}}$ for all $g \in G$. Let $P$ be the projection of the Hilbert space $\mathcal{K}$ onto $[V\mathcal{H}]$. Then for $h_1, h_2, \in \mathcal{H}$ we have
\begin{align*}
\big \langle P T V h_1, V h_2 \big \rangle &= \big \langle P T \pi(1_\mathcal{A}) V h_1, V h_2 \big \rangle \\
&= \big \langle T \pi(1_\mathcal{A}) V h_1, V h_2 \big \rangle \\
&= \big \langle V^* T \pi(1_\mathcal{A}) V h_1,  h_2 \big \rangle \\
&= \big \langle \lambda h_1,  h_2 \big \rangle \\
&= \lambda \big \langle V^* V h_1,  h_2 \big \rangle \\
&=  \big \langle \lambda V h_1, V h_2 \big \rangle.
\end{align*}
Since $h_1, h_2, \in \mathcal{H}$ were arbitrarily chosen we obtain that $PT\big |_{[V\mathcal{H}]} = \lambda 1_{[V\mathcal{H}]}$. This implies that $P\big (T - \lambda 1_\mathcal{K} \big )\big |_{[V\mathcal{H}]} = 0$. Now by using the assumption that the subspace $[V\mathcal{H}]$ is a faithful subspace of the Hilbert space $\mathcal{K}$ for the set $\mathcal{C}^S_\phi$ and observing the fact that $T - \lambda 1_\mathcal{K} \in \mathcal{C}^S_\phi$, we get  $T = \lambda 1_K$ (see Definition \ref{def; faithful subspace}). Hence, $\lambda \phi_1(\cdot) = V^* T \pi(\cdot) V  = V^* \lambda \pi(\cdot) V = \lambda \phi(\cdot)$, and thus $\phi_1 = \phi\) and a similar argument shows that \(\phi = \phi_2$. This shows that $\phi \in \text{ext} \big (\ucpapg \big )$.

Conversely, let us assume that $\phi$ is an extreme point of $\ucpapg$ and define a positive linear map $\mu: \pi(\mathcal{A})^\prime \rightarrow \mathcal{B}([V\mathcal{H}])$ by 
\begin{equation*}
\mu (T) := PT\big |_{[V\mathcal{H}]},    
\end{equation*}
where $P$ be the projection of the Hilbert space $\mathcal{K}$ onto $[V\mathcal{H}]$. Observe that $\mathcal{C}^S_\phi \subseteq \pi(\mathcal{A})^\prime$. Let $T \in \mathcal{C}^S_\phi$ be such that $\mu(T)= 0$. We need to show that \(T=0\). Since $T \in \mathcal{C}^S_\phi$, clearly $T$ is a self-adjoint operator. Thus we may chose positive scalars $s$ and $t$ such that $\frac{1}{4} 1_\mathcal{K} \leq sT + t 1_\mathcal{K} \leq \frac{3}{4} 1_\mathcal{K}$. Since $\mu$ is positive, we get $\frac{1}{4} \mu(1_\mathcal{K}) \leq \mu(sT + t 1_\mathcal{K}) \leq \frac{3}{4} \mu(1_\mathcal{K})$. This gives $\frac{1}{4} 1_{[V\mathcal{H}]} \leq  t 1_{[V\mathcal{H}]}\leq \frac{3}{4} 1_{[V\mathcal{H}]}$ and hence $0 < t < 1$. Next define
\begin{equation*}
\phi_1(\cdot) := V^\ast \big (sT + t 1_\mathcal{K} \big ) \pi(\cdot) V \; \; \; \text{and} \; \; \;  \phi_2(\cdot) := V^\ast \big (1_\mathcal{K} - (sT + t 1_\mathcal{K}) \big ) \pi(\cdot) V.
\end{equation*} 
By applying Theorem~\ref{thm;arnd}, we know that $\phi_1, \phi_2 \in [0, \phi]$ and  $\phi_1 + \phi_2 = \phi$. Let $g \in G$ and $a \in \mathcal{A}_{g^{-1}}$. Then assuming $T = \alpha_1T_1 + \alpha_2 T_2$ for some $T_1$ and $T_2$ as in Lemma \ref{lem;rnd} and $\alpha_1, \alpha_2 \in [-1, 1]$ (see Equation~\eqref{eqn; cs phi}), we have 
\begin{align*}
\phi_1 ( a ) &= V^\ast \big (sT + t 1_\mathcal{K} \big ) \pi(a) V  \\
&= V^\ast sT  \pi(a) V + V^\ast  t 1_\mathcal{K}  \pi(a) V  \\
&= s V^\ast (\alpha_1T_1 + \alpha_2 T_2) \pi(a) V + t V^\ast \pi(a) V \\
&= s\alpha_1 V^\ast T_1 \pi(a) V + s\alpha_2 V^\ast  T_2 \pi(a) V + t V^\ast \pi(a) V \\
& = s\alpha_1 V^\ast T_1 \pi(\tau_g(a)) V + s\alpha_2 V^\ast  T_2 \pi(\tau_g(a)) V + t V^\ast \pi(\tau_g(a)) V \\
&= s V^\ast (\alpha_1T_1 + \alpha_2 T_2) \pi(\tau_g(a)) V + t V^\ast \pi(\tau_g(a)) V \\
&= V^\ast sT  \pi(\tau_g(a)) V + V^\ast  t 1_\mathcal{K}  \pi(\tau_g(a)) V  \\
&= V^\ast \big (sT + t 1_\mathcal{K} \big ) \pi(\tau_g(a)) V  \\
&= \phi_1 \big ( \tau_g(a) \big)
\end{align*}
 and similarly 
\begin{equation*}
\phi_2 ( a ) = \phi_2 \big ( \tau_g(a) \big).
\end{equation*}
Thus $\phi_1, \phi_2 \in [0, \phi] \cap \cpapg$. Since $\mu(T) = 0$, we observe that
\begin{equation*}
\phi_1 \big (1_\mathcal{A} \big ) = V^\ast \big (sT + t 1_\mathcal{K} \big ) \pi\big (1_\mathcal{A} \big ) V = sV^\ast T V + tV^\ast V = PT\big |_{[V\mathcal{H}]} + t1_\mathcal{H} =  t1_\mathcal{H}
\end{equation*}
and similarly $\phi_2 \big (1_\mathcal{A} \big ) = (1-t) 1_\mathcal{H}$. This implies 
\begin{equation*}
\frac{1}{t} \phi_1, \; \; \frac{1}{1 - t} \phi_2 \in \ucpapg.  
\end{equation*}
Recall that $\phi = \phi_1 + \phi_2$ and write $\phi = t (\frac{1}{t} \phi_1) + (1 - t) (\frac{1}{1 - t} \phi_2)$. Then the assumption that $\phi$ is an extreme point implies that 
\begin{equation*}
\phi = \frac{1}{t} \phi_1 = \frac{1}{1 - t} \phi_2.
\end{equation*}
This gives 
\begin{equation*}
\phi_1(\cdot) = V^\ast \big (sT + t 1_\mathcal{K} \big ) \pi(\cdot) V = V^\ast t1_\mathcal{K} \pi(\cdot) V = t\phi \leq \phi.
\end{equation*}
By the order preserving affine isomorphism from Theorem~\ref{thm;arnd}, we must have $sT + t 1_\mathcal{K} = t 1_\mathcal{K}$, which implies that $T=0$ (as $s, t > 0$). This shows that the map $\mu$ is injective on $\mathcal{C}^S_\phi$ and therefore the subspace $[V\mathcal{H}]$ is a faithful subspace of the Hilbert space $\mathcal{K}$ for the set $\mathcal{C}^S_\phi$. 

Finally, we prove $(1) \iff (4)$. Let $\big ( \mathcal{X}, \sigma, e \big )$ be the Paschke's dilation of $\phi$ be such that $\phi(\cdot) = \big \langle \sigma(\cdot)e, e \big \rangle$ as given in Theorem \ref{thm;pd}. By Theorem \ref{thm;prnd}, we get 
\begin{equation} \label{eqn; phi using pd}
\phi(\cdot) = \big \langle \sigma(\cdot)e, e \big \rangle = \big \langle 1_{\mathcal{X}^\prime} \widetilde{\sigma}(\cdot)\hat{e}, \hat{e} \big \rangle.
\end{equation}
First we assume that if any $T \in \mathcal{C}^P_\phi$ satisfies $\langle T \hat{e}, \hat{e} \rangle = 0$, then $T = 0$. Assuming this, we show that $\phi \in \text{ext} \big (\ucpapg \big )$. Let $\phi_1, \phi_2 \in \ucpapg$, and $0 < \lambda < 1$ be such that $\phi = \lambda \phi_1 + (1- \lambda) \phi_2$. As $\lambda \phi_1 \leq \phi$ and $\lambda \phi_1 \in \cpapg$ by following Lemma \ref{lem;prnd} we get 
\begin{equation} \label{eqn; lambda phi1 using pd}
\lambda \phi_1 (\cdot) = \big \langle T \widetilde{\sigma}(\cdot) \hat{e}, \hat{e} \big  \rangle
\end{equation}
for some $T \in  \widetilde{\sigma}(\mathcal{A})^\prime$ such that $0 \leq T \leq 1_{\mathcal{X}^\prime}$ satisfying the properties $T(\mathcal{Y}_g) \subseteq \mathcal{Y}_g$ and $T\widetilde{W}_g \big |_{\mathcal{Y}_{g^{-1}}} = \widetilde{W}_g T \big |_{\mathcal{Y}_{g^{-1}}}$  for all $g \in G$ (see Equations~\eqref{eqn; Y g inverse}, and~\eqref{eqn; W phi tilde}). By following Equations~\eqref{eqn; phi using pd} and~\eqref{eqn; lambda phi1 using pd}, we obtain
\begin{equation*}
 \lambda \phi_1 \big (1_\mathcal{A} \big ) = \big \langle T \widetilde{\sigma} \big (1_\mathcal{A} \big ) \hat{e}, \hat{e} \big \rangle = \lambda 1_\mathcal{H}   = \lambda \big \langle 1_{\mathcal{X}^\prime} \widetilde{\sigma}\big (1_\mathcal{A} \big ) \hat{e}, \hat{e} \big \rangle = \lambda \phi\big (1_\mathcal{A} \big ).
\end{equation*}
This implies $\big \langle (T - \lambda 1_{\mathcal{X}^\prime}) \widetilde{\sigma} \big (1_\mathcal{A} \big ) \hat{e}, \hat{e} \big \rangle = \big \langle (T - \lambda 1_{\mathcal{X}^\prime})\hat{e}, \hat{e} \big \rangle = 0$. Now by using the assumption and observing the fact that $T - \lambda 1_{\mathcal{X}^\prime} \in \mathcal{C}^P_\phi$, we get  $T =  \lambda 1_{\mathcal{X}^\prime}$. Hence, $\lambda \phi_1(\cdot) = \big \langle \lambda \widetilde{\sigma}(\cdot)\hat{e}, \hat{e} \big \rangle = \lambda \phi(\cdot)$ and thus $\phi_1 = \phi = \phi_2$. This shows that $\phi \in \text{ext} \big (\ucpapg \big )$.

Conversely, let us assume that $\phi$ is an extreme point of $\ucpapg$ and define a positive linear map $\mu : \widetilde{\sigma}(\mathcal{A})^\prime \rightarrow \mathcal{B}\big (\mathcal{H} \big )$ by 
\begin{equation*}
S \mapsto \big \langle S \hat{e}, \hat{e} \big \rangle \; \; \; \text{for all} \; \; S \in  \widetilde{\sigma}(\mathcal{A})^\prime.
\end{equation*}
We observe that $\mathcal{C}^P_\phi \subseteq \widetilde{\sigma}(\mathcal{A})^\prime$. Let $T \in \mathcal{C}^P_\phi$ be such that $\mu(T) = \big \langle T \hat{e}, \hat{e} \big \rangle = 0$. We need to show that \(T=0\). Since $T \in \mathcal{C}^P_\phi$, clearly $T$ is a self-adjoint operator. Thus we may choose positive scalars $s, t$ such that $\frac{1}{4} 1_{\mathcal{X}^\prime} \leq sT + t 1_{\mathcal{X}^\prime} \leq \frac{3}{4} 1_{\mathcal{X}^\prime}$. Since $\mu$ is positive and $\mu(T) = 0$, we get $\frac{1}{4} \big \langle 1_{\mathcal{X}^\prime}\hat{e}, \hat{e} \big \rangle \leq  t \big \langle 1_{\mathcal{X}^\prime}\hat{e}, \hat{e} \big \rangle  \leq \frac{3}{4} \big \langle 1_{\mathcal{X}^\prime}\hat{e}, \hat{e} \big \rangle$. Hence $0 < t < 1$. Next for $i=1, 2$ define $\phi_i : \mathcal{A} \rightarrow \mathcal{B}\big (\mathcal{H} \big )$ by
\begin{equation*}
\phi_1(\cdot) := \big \langle \big (sT + t 1_{\mathcal{X}^\prime} \big ) \widetilde{\sigma}(\cdot) \hat{e}, \hat{e} \big \rangle  \hspace{10pt} \text{and} \hspace{10pt}
\phi_2(\cdot) := \big \langle \big (1_{\mathcal{X}^\prime} - (sT + t 1_{\mathcal{X}^\prime}) \big) \widetilde{\sigma}(\cdot) \hat{e}, \hat{e} \big \rangle .
\end{equation*}
By applying Theorem \ref{thm;prnd}, we know that $\phi_1, \phi_2 \in [0, \phi]$ and  $\phi_1 + \phi_2 = \phi$. Let $g \in G$ and $a \in \mathcal{A}_{g^{-1}}$. Then assuming $T = \alpha_1T_1 + \alpha_2 T_2$ for some $T_1$ and $T_2$ as in Lemma \ref{lem;prnd} and $\alpha_1, \alpha_2 \in [-1, 1]$ (see Equation~\eqref{eqn; cp phi}), we have 
\begin{align*}
\phi_1 \big ( a \big) &= \big \langle \big (sT + t 1_{\mathcal{X}^\prime} \big ) \widetilde{\sigma} \big ( a \big) \hat{e}, \hat{e} \big \rangle  \\
&= \big \langle sT \widetilde{\sigma} \big ( a \big) \hat{e}, \hat{e} \big \rangle + \big \langle  t 1_{\mathcal{X}^\prime}  \widetilde{\sigma} \big ( a \big) \hat{e}, \hat{e} \big \rangle  \\
&= s \big \langle (\alpha_1T_1 + \alpha_2 T_2) \widetilde{\sigma} \big ( a \big) \hat{e}, \hat{e} \big \rangle + t \big \langle  1_{\mathcal{X}^\prime}  \widetilde{\sigma} \big ( a \big) \hat{e}, \hat{e} \big \rangle \\
&= s\alpha_1 \big \langle T_1 \widetilde{\sigma} \big ( a \big) \hat{e}, \hat{e} \big \rangle + s\alpha_2 \big \langle  T_2 \widetilde{\sigma} \big ( a \big) \hat{e}, \hat{e} \big \rangle + t \big \langle  \widetilde{\sigma} \big ( a \big) \hat{e}, \hat{e} \big \rangle \\
& = s\alpha_1 \big \langle T_1 \widetilde{\sigma} \big ( \tau_g(a) \big) \hat{e}, \hat{e} \big \rangle + s\alpha_2 \big \langle  T_2 \widetilde{\sigma} \big ( \tau_g(a) \big) \hat{e}, \hat{e} \big \rangle + t \big \langle  \widetilde{\sigma} \big ( \tau_g(a) \big) \hat{e}, \hat{e} \big \rangle  \\
&= s \big \langle (\alpha_1T_1 + \alpha_2 T_2) \widetilde{\sigma} \big ( \tau_g(a) \big) \hat{e}, \hat{e} \big \rangle + t \big \langle  1_{\mathcal{X}^\prime}  \widetilde{\sigma} \big ( \tau_g(a) \big) \hat{e}, \hat{e} \big \rangle \\
&= \big \langle sT \widetilde{\sigma} \big ( \tau_g(a) \big) \hat{e}, \hat{e} \big \rangle + \big \langle  t 1_{\mathcal{X}^\prime}  \widetilde{\sigma} \big ( \tau_g(a) \big) \hat{e}, \hat{e} \big \rangle   \\
&= \big \langle \big (sT + t 1_{\mathcal{X}^\prime} \big ) \widetilde{\sigma} \big ( \tau_g(a) \big) \hat{e}, \hat{e} \big \rangle  \\
&= \phi_1 \big ( \tau_g(a) \big)
\end{align*}
and similarly 
\begin{equation*}
\phi_2 \big ( a \big) = \phi_2 \big ( \tau_g(a) \big).
\end{equation*}
Thus $\phi_1, \phi_2 \in [0, \phi] \cap \cpapg$. Since $\mu(T) = \big \langle T \hat{e}, \hat{e} \big \rangle = 0$, we observe that
\begin{equation*}
\phi_1 \big ( 1_\mathcal{A} \big) = \big \langle \big (sT + t 1_{\mathcal{X}^\prime} \big ) \widetilde{\sigma} \big ( 1_\mathcal{A} \big) \hat{e}, \hat{e} \big \rangle = s \big \langle T \hat{e}, \hat{e} \big \rangle + t \big \langle 1_{\mathcal{X}^\prime}  \hat{e}, \hat{e} \big \rangle = 0 + t1_\mathcal{H} =  t1_\mathcal{H}
\end{equation*}
and similarly $\phi_2 \big (1_\mathcal{A} \big ) = (1-t) 1_\mathcal{H}$. This implies 
\begin{equation*}
\frac{1}{t} \phi_1, \; \; \frac{1}{1 - t} \phi_2 \in \ucpapg.  
\end{equation*}
Recall that $\phi = \phi_1 + \phi_2$ and write $\phi = t (\frac{1}{t} \phi_1) + (1 - t) (\frac{1}{1 - t} \phi_2)$. Then the assumption that $\phi$ is an extreme point implies that 
\begin{equation*}
\phi = \frac{1}{t} \phi_1 = \frac{1}{1 - t} \phi_2.
\end{equation*}
This gives 
\begin{equation*}
\phi_1(\cdot) := \big \langle \big (sT + t 1_{\mathcal{X}^\prime} \big ) \widetilde{\sigma}(\cdot) \hat{e}, \hat{e} \big \rangle  = \big \langle  t 1_{\mathcal{X}^\prime} \widetilde{\sigma}(\cdot) \hat{e}, \hat{e} \big \rangle = t\phi \leq \phi.
\end{equation*}
By the order preserving affine isomorphism from Theorem \ref{thm;prnd} we must have $sT + t 1_{\mathcal{X}^\prime} = t 1_{\mathcal{X}^\prime}$, which in turn implies that $T=0$ (as $s, t > 0$). This proves the result.
\end{proof}

We conclude this section with the following remark, which provides a decomposition of elements in $\ucpapg$ using the classical Choquet theorem in view of the description of the set of extreme points of $\ucpapg$ as given in Theorem \ref{thm;egiucp}.  

\begin{remark}
Let $G$ be a group, $\mathcal{H}$ be a Hilbert space and $\mathcal{A}$ be a unital $C^\ast$-algebra. Let $\tau = \big ( \{ \mathcal{A}_g \}_{g \in G},  \{ \tau_g \}_{g \in G}  \big )$ be a partial action of $G$ on $\mathcal{A}$. By using the fact that every completely positive map is completely bounded (see~\cite[Proposition 3.6]{Paulsen}), we have the following set containments
\begin{equation*}
\ucpapg \subseteq \cpa \subseteq \cba.
\end{equation*}
This gives that the set $\ucpapg$ is a compact (in BW-topology) and convex subset of the locally convex topological space $\cba$. We recall the classical Choquet theorem, which gives a \textit{barycentric decomposition} in a compact and convex subset of a locally convex topological vector space. 
\begin{theorem} [{\cite[Page 14]{Phlp}}] \label{thm; Choquet}
Suppose that $X$ is a metrizable compact convex subset of a locally convex space $E$ and that $x_0$ is an element of $X$. Then there is a probability measure $\mu$ on $X$ whose \textit{barycenter} is $x_0$, that is,
\begin{equation} \label{eqn; barycentric decomposition 2}
f(x_0) = \int_{X} f \, \mathrm{d} \mu   
\end{equation}
for every continuous linear functional $f$ on $E$ and $\mu$ is supported by the extreme points of $X$.
\end{theorem}
\noindent
The decomposition of $x_0$ obtained in Equation \eqref{eqn; barycentric decomposition 2} is a \textit{barycentric decomposition} of $x_0$. If $\mathcal{A}$ is a separable $C^\ast$-algebra, then the set $\ucpapg$ is a metrizable set. In this case, we obtain the barycentric decomposition of elements in $\ucpapg$ by using Theorem \ref{thm; Choquet}. In Theorem~\ref{thm; Choquet}, we replace $X = \ucpapg$ and $E= \cba$, and we obtain: given $\phi_0 \in \ucpapg$, there exists a probability measure $\mu$ on $\ucpapg$ whose barycenter is $\phi_0$, that is,  
\begin{equation*}
f(\phi_0) = \int_{X} f \, \mathrm{d} \mu   
\end{equation*}
for every continuous linear functional $f$ on $\cba$ and $\mu$ is supported by the extreme points of $\ucpapg$. Then we invoke Theorem~\ref{thm;egiucp}, which gives the characterisation of extreme points of set $\ucpapg$, and this completes the picture of barycentric decomposition in the space $\ucpapg$, when $\mathcal{A}$ is separable. Next, we recall the Choquet-Bishop-dee Leeuw theorem when $X$ is not a metrizable set.

\begin{theorem} [{\cite[Page 17]{Phlp}}] \label{thm; Choquet Bishop dee leeuw}
Suppose that $X$ is a compact convex subset of a locally convex space $E$ and that $x_0$ is an element of $X$. Then there is a probability measure $\mu$ on $X$ whose \textit{barycenter} is $x_0$, that is,
\begin{equation*}
f(x_0) = \int_{X} f \, \mathrm{d} \mu   
\end{equation*}
for every continuous linear functional $f$ on $E$ and $\mu$ vanishes on every Baire subset of $X$ which is disjoint from the set of extreme points of $X$.
\end{theorem}
\noindent
For the definition of Baire sets, one may refer to \cite[Section 4.1.2]{OB1} or \cite[Chapter 4]{Phlp}.  If $\mathcal{A}$ is not a separable $C^\ast$-algebra, then the set $\ucpapg$ need not be a metrizable set. In this case, we obtain the barycentric decomposition of elements in $\ucpapg$ by using Theorem \ref{thm; Choquet Bishop dee leeuw}. We replace $X = \ucpapg$ and $E= \cba$ as described above and then make use of Theorem \ref{thm;egiucp} to complete the picture of barycentric decomposition in the space $\ucpapg$, when $\mathcal{A}$ is not separable.
\end{remark}

\subsection*{Acknowledgements}
The first-named author kindly acknowledges the financial support received as an Institute postdoctoral fellowship from the Indian Institute of Science Education and Research Mohali. 

\subsection*{Availability of data and material} Not applicable.
\subsection*{Competing interests} The authors declare that there are no conflicts of interest.
%\subsection*{Declaration}
%The authors declare that there are no conflicts of interest.

\bibliographystyle{plain}
%\bibliography{mybib}   

\end{document}